\documentclass[12pt,a4]{amsart}
\usepackage{amsthm,amsmath,amssymb}
\usepackage{titlesec}
\allowdisplaybreaks[1]
\makeatletter
\def\section{\@startsection{section}{1}%
  \z@{.7\linespacing\@plus\linespacing}{.5\linespacing}%
  {\normalfont\scshape\centering}}
\def\subsection{\@startsection{subsection}{2}%
  \z@{.5\linespacing\@plus.7\linespacing}{-.5em}%
  {\normalfont\bfseries}}
\makeatother
\titleformat*{\section}{\large\bfseries}
\titleformat*{\subsection}{\large\bfseries}
\newtheorem{theorem}{Theorem}[section]

\newtheorem{lemma}{Lemma}[section]

\theoremstyle{remark}
\newtheorem{rem}{Remark}
\renewcommand{\labelenumi}{{\upshape(\roman{enumi})}}
\title[a certain sum of the Derivatives of Dirichlet $L$-Functions]{On a certain sum of the Derivatives of Dirichlet $L$-Functions}
\author{Hirotaka Kobayashi}
\date{}
\address{Graduate School of Mathematics, Nagoya University, Furocho, Chikusaku, Nagoya 464-8602, Japan}
\email{m17011z@math.nagoya-u.ac.jp}
\begin{document}

\maketitle

\begin{abstract}
We consider a sum of the derivatives of Dirichlet $L$-functions over the zeros of Dirichlet $L$-functions.
We give an asymptotic formula for the sum.
\end{abstract}

\section{Introduction}

Let $s=\sigma+it$ denote a complex variable. The Dirichlet $L$-function attached to $\chi$ is defined by
\begin{equation*}
L(s,\chi)=\sum_{n=1}^{\infty}\frac{\chi(n)}{n^s} \quad (\sigma>1),
\end{equation*}
where $\chi(n)$ is a Dirichlet character modulo $q$. For $\chi \pmod 1$ we get the Riemann $\zeta$-function $L(s,\chi)=\zeta(s)$.
The Generalized Riemann Hypothesis (GRH) states that all zeros of every Dirichlet $L$-function in the strip $0<\sigma<1$ lie on the line $\sigma=1/2$.
We denote the zeros in the strip $0<\sigma<1$ by $\rho_{\chi}=\beta_{\chi}+i\gamma_{\chi}$.
A Dirichlet character is said to be primitive when it is not induced by any other character of modulus strictly less than $q$.
The unique principal character modulo $q$ is denoted by $\chi_0$.
When $\chi=\chi_0$, we have $L(s,\chi_0)=\zeta(s)\prod_{p\mid q}(1-p^{-s})$,
where, and in what follows, $p$ denotes a prime number.
For a Dirichlet character $\chi \pmod q$ the Gauss sum is defined by
\begin{equation*}
\tau(\chi)=\sum_{a=1}^{q}\chi(a)\exp \left(2\pi i\frac{a}{q} \right).
\end{equation*}
For a primitive character $\chi \pmod q$ we have $|\tau(\chi)|=\sqrt{q}$.

In this paper, $T$ is a positive number which always tends to $+\infty$ and $\varepsilon>0$.
Our main theorem is

\begin{theorem}\label{Th}
Let $c_1$ be a positive constant. Let $\chi \pmod q$ be a primitive character. Then, uniformly for $q \leq \exp(c_1\sqrt{\log T})$, we have
\begin{align*}
\sum_{0<\gamma_{\chi} \leq T} L'(\rho_{\chi},\chi)
&=\frac{1}{4\pi}T\left(\log\frac{qT}{2\pi}\right)^2+a_1\frac{T}{2\pi}\log \frac{qT}{2\pi}+a_2\frac{T}{2\pi}+a_3 \\
&\quad+O\left(T\exp \left(-c\sqrt{\log T}\right)\right),
\end{align*}
where the implicit constant is absolute, $c$ is a positive absolute constant depends on $c_1$ and
\begin{equation*}
a_1=\sum_{p\mid q}\frac{\log p}{p-1}+\gamma_0-1,
\end{equation*}

\begin{align*}
a_2&=\frac{1}{2}\left(\sum_{p\mid q}\frac{\log p}{p-1} \right)^2+(\gamma_0-1)\sum_{p\mid q}\frac{\log p}{p-1} \\
&\quad-\frac{3}{2}\sum_{p\mid q}p\left(\frac{\log p}{p-1}\right)^2+1-\gamma_0-\gamma_0^2+3\gamma_1
\end{align*}
with the stieltjes constants $\gamma_0, \gamma_1$ and
\begin{equation*}
a_3=\frac{\omega \chi(-1)\tau({\overline{\chi})}\tau(\overline{\omega}\chi)}{q\varphi(q)}\frac{L'(\beta,\omega)}{\beta}\left(\frac{qT}{2\pi}\right)^{\beta}
\end{equation*}
when $L(s,\omega)$ with a quadratic character $\omega \pmod q$ has an exceptional zero $\beta$, otherwise $a_3=0$.

Assuming the GRH, we can replace the error term by $(qT)^{\frac{1}{2}+\varepsilon}$ uniformly for $q \ll T^{1-\varepsilon}$.
\end{theorem}

\begin{rem}
Let $q$ be a prime power.
If we could obtain the estimate
\begin{equation}\label{L'-upper}
\sum_{\gamma_{\chi}\leq T}\left|L'(\rho_{\chi},\chi) \right|^2 \ll T(\log qT)^4,
\end{equation}
where the implicit constant is absolute, we could replace the error term by $\sqrt{qT}(\log qT)^{\frac{7}{2}}$ under the GRH.
We will give the details at the last section. In view of Gonek's formula (\ref{zeta'-upper}), the above estimate (\ref{L'-upper}) may be plausible.
\end{rem}

When $q=1$, the above theorem implies Fujii's Theorem 1 in \cite{F}.
Our proof is a generalization of his method.
However, it is not easy to obtain his Theorem 2 in \cite{F} and we give a weaker statement.
Kaptan, Karabulut and Y\i ld\i r\i m \cite{K-K-Y} consider more general cases and give the asymptotic formula, that is for $\mu \geq 1$ and $q\leq (\log T)^A$ with any fixed $A>0$
\begin{equation*}
\sum_{0\leq \gamma_{\chi}\leq T}L^{(\mu)}(\rho_{\chi},\chi)=\frac{(-1)^{\mu}}{\mu+1}\frac{T}{2\pi}\left(\log \frac{qT}{2\pi} \right)^{\mu+1}+O(T(\log T)^{\mu+\varepsilon})
\end{equation*}
for any fixed $\varepsilon>0$. Our result is the case $\mu=1$ in their paper and gives a more sophisticated formula.
Jakhlouti and Mazhouda \cite{J-M} consider the sum
\begin{equation*}
\sum_{\substack{\rho_{a,\chi} \\ 0<\gamma_{a,\chi}\leq T}}L'(\rho_{a,\chi},\chi)X^{\rho_{a,\chi}},
\end{equation*}
where $\rho_{a,\chi}=\beta_{a,\chi}+i\gamma_{a,\chi}$ are the zeros of $L(s,\chi)-a$ for any fixed complex number $a$ and $X$ is a fixed positive number.
They also fix $\chi$ throughout their paper.
Hence our main theorem treats a special case of their sum, but our result gives a more precise form because we do not fix $\chi$.

\section{Preliminaries}

The Dirichlet $L$-function attached to a primitive character $\chi \pmod q$ satisfies the functional equation

\begin{equation}\label{fe}
L(s,\chi)=\Delta(s,\chi)L(1-s,\overline{\chi}),
\end{equation}
where
\begin{equation*}
\Delta(s,\chi)=\varepsilon(\chi) 2^s \pi^{s - 1} q^{\frac{1}{2} - s}\Gamma (1 - s) \sin \frac{\pi}{2} (s + \kappa)
\end{equation*}
when we put

\begin{equation*}
\kappa=\frac{1-\chi(-1)}{2}
\end{equation*}
and

\begin{equation*}
\varepsilon (\chi) = \frac{\tau (\chi)}{i^{\kappa} \sqrt{q}}.
\end{equation*}
We note that $\Delta(s,\chi)$ is a meromorphic function with only real zeros and poles satisfying the functional equation
\begin{equation*}
\Delta(s,\chi)\Delta(1-s,\overline{\chi})=1.
\end{equation*}
By Stirling's formula, we can show that

\begin{lemma}\label{Stirling}
For $-1\leq\sigma\leq2$ and $t \geq 1$, we have

\begin{equation}
\Delta(1-s,\chi)=\frac{\tau(\chi)}{\sqrt{q}}e^{-\frac{\pi i}{4}}\left(\frac{qt}{2\pi} \right)^{\sigma-\frac{1}{2}}\exp \left(it\log \frac{qt}{2\pi e}\right) \left(1+O\left(\frac{1}{t}\right)\right)
\end{equation}
and

\begin{equation}\label{Stirling-log-der}
\frac{\Delta'}{\Delta}(s,\chi)=-\log \frac{qt}{2\pi}+O\left(\frac{1}{t}\right).
\end{equation}
\end{lemma}

A theorem from \cite{H-B} and an application of the Phragmen-Lindel\"{o}f principle yields the estimate

\begin{equation}\label{convex}
L(s,\chi) \ll (q(|t|+2))^{\frac{3}{16}+\varepsilon} \quad \text{for} \quad \frac{1}{2} \leq \sigma \leq 1+\frac{1}{\log qT},
\end{equation}

\begin{equation}\label{convex2}
L(s,\chi) \ll (q(|t|+2))^{\frac{1}{2}}\log q(|t|+2) \quad \text{for} \quad -\frac{1}{\log qT}\leq \sigma < \frac{1}{2}
\end{equation}
uniformly in $|t|\ll T$ for any non-principal Dirichlet character $\chi \pmod q$.
When we assume the GRH, the bound of (\ref{convex}) can be replaced by $(q(|t|+2))^{\varepsilon}$.
For the principal character, we need the restriction $|s-1|\gg 1$ in (\ref{convex}).
For the logarithmic derivative it is known that for $q\geq 1$ and $\chi \pmod q$
\begin{equation}\label{log-der-sum}
\frac{L'}{L}(s,\chi)=\sum_{|t-\gamma_{\chi}|\leq 1}\frac{1}{s-\rho_{\chi}}+O(\log q(|t|+2)) \quad \text{for} \quad -1\leq \sigma \leq 2, \ |t|\geq 1
\end{equation}
(see \cite[p. 225]{Pra}).
For $q\geq 1$, $\chi \pmod q$ and $t\geq 0$ we have (see \cite[p. 220]{Pra})
\begin{equation}\label{non-dens-zero}
\begin{split}
N(t+1,\chi)-N(t,\chi)
&:= \# \{ \rho_{\chi}=\beta_{\chi}+i\gamma_{\chi} : t<\gamma_{\chi}\leq t+1 \} \\
&\ll \log q(t+2).
\end{split}
\end{equation}
Hence for any $T_0 \geq 0$, there exists a $t=t(\chi), \ t \in (T_0, T_0+1]$, such that
\begin{equation}\label{awayzero}
\min_{\gamma_{\chi}}|t-\gamma_{\chi}|\gg \frac{1}{\log q(t+2)}.
\end{equation}

By the expression (\ref{log-der-sum}), it follows that for $q\geq 1$, $\chi \pmod q$ and $t$ satisfying (\ref{awayzero})

\begin{equation}\label{log-der-bd}
\frac{L'}{L}(\sigma+it,\chi)\ll (\log q(|t|+2))^2 \quad \text{for} \quad -1\leq \sigma \leq 2
\end{equation}
uniformly. This estimate is valid for $|s-\rho_{\chi}|\gg (\log (q(|t|+2)))^{-1}$ though $t$ is not satisfying (\ref{awayzero}).

We will apply the following approximate functional equation for $L(s,\chi)$.
\begin{lemma}[A. F. Lavrik \cite{A.F.Lav}]\label{afe}
We let $0\leq \sigma \leq 1$, $2\pi xy=t$, $x\geq 1$ and $y\geq 1$.
Then for $t>0$, we get
\begin{align*}
L(s,\chi)&=\sum_{n\leq x}\frac{\chi(n)}{n^s}+\Delta(s,\chi)\sum_{n\leq y}\frac{\overline{\chi}(n)}{n^{1-s}} \\
&\quad+O\left(\sqrt{q}\left(y^{-\sigma}+x^{\sigma-1}(qt)^{\frac{1}{2}-\sigma}\right)\log 2t \right).
\end{align*}
\end{lemma}

On the other hand, for $t>t_0>0$ and $\sigma>1$, using partial summation, we get
\begin{equation}\label{partialsum}
L(s,\chi)=\sum_{n\leq qt}\frac{\chi(n)}{n^s}+O\left(\frac{q|s|}{\sigma}(qt)^{-\sigma}\right).
\end{equation}

We will use the following modified Gonek's lemma (\cite[Lemma 5]{Go}).
\begin{lemma}\label{gonek}
Let $\{b_n\}_{n=1}^{\infty}$ be a sequence of complex numbers such that $b_n \ll n^{\varepsilon}$ for any $\varepsilon >0$. Let $a>1$ and let $m$ be a non-negative integer. Then for any sufficiently large $T$,
\begin{align*}
&\quad \frac{1}{2\pi}\int_{1}^{T}\left(\sum_{n=1}^{\infty}\frac{b_n}{n^{a+it}} \right)\Delta(1-a-it,\chi)\left(\log \frac{qt}{2\pi} \right)^m dt \\
&=\frac{\tau(\chi)}{q}\sum_{1\leq n \leq qT/2\pi} b_ne\left(-\frac{n}{q} \right)(\log n)^m+O\left(\left|\sum_{n=1}^{\infty}\frac{b_n}{n^a}\right|(qT)^{a-1/2}(\log qT)^m\right).
\end{align*}
\end{lemma}
This is provided implicitly by Steuding in \cite{Ste}.
\section{Proof in the unconditional case}
In this section we prove the claim of the unconditional part of Theorem \ref{Th}.
Let $(\log 2q)^{-1}\ll b \leq 1$ and $T\geq 2$ be such that

\begin{equation*}
\min_{\gamma_{\chi}}|b-\gamma_{\chi}|\gg \frac{1}{\log 2q} \quad \text{and} \quad \min_{\gamma_{\chi}}|T-\gamma_{\chi}|\gg \frac{1}{\log qT}.
\end{equation*}

We prove the theorem under this situation. At the end of the proof, we remove this restriction.
Let $a=1+(\log qT)^{-1}$ and define the contour $C$ as the positively oriented rectangular path with vertices $a+ib, a+iT, 1-a+iT$ and $1-a+ib$.
By the residue theorem, our sum can be written as a contour integral

\begin{equation*}
\sum_{0<\gamma_{\chi} \leq T} L'(\rho_{\chi},\chi)=\frac{1}{2\pi i}\int_{C}\frac{L'}{L}(s,\chi)L'(s,\chi)ds+E,
\end{equation*}
where $E$ consists of the terms $L'(\rho_{\chi},\chi)$ with $0<\gamma_{\chi}<b$.

For zeros $\rho_{\chi}=\beta_{\chi}+i\gamma_{\chi}$ with $0<\gamma_{\chi}<b$ we have
\begin{equation*}
L'(\rho_{\chi},\chi)\ll q^{\frac{1}{2}}(\log 2q)^2
\end{equation*}
by (\ref{convex}), (\ref{convex2}) and the Cauchy's integral formula applied to the circle with centre $\rho_{\chi}$ and radius $(\log 2q)^{-1}$. Therefore, by (\ref{non-dens-zero}), we have
\begin{equation*}
E=\sum_{0<\gamma_{\chi}<b} L'(\rho_{\chi},\chi) \ll q^{\frac{1}{2}}(\log 2q)^2\sum_{0<\gamma_{\chi}<b}1 \ll q^{\frac{1}{2}}(\log 2q)^3.
\end{equation*}

Next we consider the contour integral
\begin{align*}
 &\quad \frac{1}{2\pi i}\int_{C}\frac{L'}{L}(s,\chi)L'(s,\chi)ds \\
 &=\frac{1}{2\pi i}\left \{ \int_{a+ib}^{a+iT}+\int_{1-a+iT}^{1-a+ib}+\int_{a+iT}^{1-a+iT}+\int_{1-a+ib}^{a+ib} \right \}\frac{L'}{L}(s,\chi)L'(s,\chi)ds \\
 &=I_1+I_2+I_3+I_4,
\end{align*}
say.

By the Laurent expansion of the Riemann $\zeta$-function, it is easily seen that

\begin{align*}
I_1 &= \frac{1}{2\pi}\int_{b}^{T}\frac{L'}{L}(a+it,\chi)L'(a+it,\chi)dt \\
    &= \frac{1}{2\pi}\sum_{m=2}^{\infty}\sum_{n=2}^{\infty}\frac{\chi(m)\Lambda(m)\chi(n)\log n}{(mn)^a}\int_{b}^{T}\frac{dt}{(mn)^{it}} \\
    &\ll \left|\frac{\zeta'}{\zeta}(a) \right||\zeta'(a)| \ll (\log qT)^3,
\end{align*}
where $\Lambda(m)$ is the von-Mangoldt function.
To estimate the integral on the horizontal line, we will show the following lemma.
\begin{lemma}\label{horizontal}
Let $\chi$ be a primitive character, then
\begin{equation*}
\int_{1-a}^{a}L'(\sigma+iT,\chi)d\sigma \ll \sqrt{qT}\log qT.
\end{equation*}
\end{lemma}

\begin{proof}
Let
\begin{equation*}
\delta=\frac{1}{\log qT}.
\end{equation*}
Then $L(w,\chi)$ is analytic on the disk $|s-w|\leq \delta$, for $s=\sigma+iT$ with $1-a\leq \sigma \leq a$. Therefore, by  Cauchy's integral formula,
\begin{align*}
L'(s,\chi)
&=
\frac{1}{2\pi i}\int_{|s-w|=\delta}\frac{L(w,\chi)}{(s-w)^2}dw\\
&\ll \log qT\int_{0}^{2\pi}|L(s+\delta e^{i\theta},\chi)|d\theta.
\end{align*}
Thus it suffices to prove that
\begin{equation*}
\int_{1-a}^{a}\int_{0}^{2\pi}|L(s+\delta e^{i\theta},\chi)|d\theta d\sigma
=\int_{0}^{2\pi}\int_{1-a}^{a}|L(s+\delta e^{i\theta},\chi)|d\sigma d\theta \ll \sqrt{qT}.
\end{equation*}
From the functional equation and, for $1-a\leq \sigma \leq 1/2$, we have

\begin{align*}
&\quad
\int_{1-a}^{\frac{1}{2}}|L(s+\delta e^{i\theta},\chi)|d\sigma \\
&=
\int_{1-a}^{\frac{1}{2}}|\Delta(s+\delta e^{i\theta},\chi)L(1-s-\delta e^{i\theta},\overline{\chi})|d\sigma \\
&=
\int_{\frac{1}{2}}^{a}\left|\Delta(1-\sigma+iT+\delta e^{i\theta},\chi)L(\sigma-iT-\delta e^{i\theta},\overline{\chi})\right|d\sigma.
\end{align*}
On the second equality, we change the variable $\sigma$ to $1-\sigma$.
Since

\begin{align*}
&\quad
\Delta(1-(\sigma-iT-\delta e^{i\theta}),\chi) \\
&=
\overline{\Delta(1-(\sigma+iT-\delta e^{-i\theta}),\overline{\chi})} \\
&=
\overline{\frac{\tau(\overline{\chi})}{\sqrt{q}}e^{\frac{\pi i}{4}}\left(\frac{qT}{2\pi} \right)^{\sigma-\delta \cos \theta-\frac{1}{2}}\exp \left(iT\log \frac{qT}{2\pi e}\right)} \left(1+O\left(\frac{1}{T}\right)\right)
\end{align*}
by Lemma \ref{Stirling}, the integral can be bounded by

\begin{equation*}
\int_{\frac{1}{2}}^{a}(qT)^{\sigma-\delta \cos \theta-\frac{1}{2}}\left|L(\sigma+iT-\delta e^{-i\theta},\chi)\right|d\sigma.
\end{equation*}
Therefore we obtain

\begin{align*}
&\quad
\int_{1-a}^{a}|L(s+\delta e^{i\theta},\chi)|d\sigma \\
&\ll
\int_{\frac{1}{2}}^{a}|L(\sigma+iT+\delta e^{i\theta},\chi)|d\sigma \\
&\quad
+\int_{\frac{1}{2}}^{a}(qT)^{\sigma-\delta \cos \theta-\frac{1}{2}}\left|L(\sigma+iT-\delta e^{-i\theta},\chi)\right|d\sigma \\
&\ll
\int_{\frac{1}{2}}^{a}(qT)^{\sigma-\frac{1}{2}}\left|L(\sigma+iT\pm \delta e^{\pm i\theta},\chi) \right|d\sigma.
\end{align*}
On the last inequality, we use the facts that

\begin{equation*}
(qT)^{\delta}=e
\end{equation*}
with $\delta=(\log qT)^{-1}$.
This integral is
\begin{align*}
&=\left \{\int_{\frac{1}{2}}^{1}+\int_{1}^{a} \right \}(qT)^{\sigma-\frac{1}{2}}\left|L(\sigma+iT\pm \delta e^{\pm i\theta},\chi) \right|d\sigma \\
&=S_1+S_2,
\end{align*}
say. Using Lemma \ref{afe}, we have
\begin{align*}
S_1&\ll (qT)^{-\frac{1}{2}}\sum_{n\ll \sqrt{qT}}n^{\delta} \int_{\frac{1}{2}}^{1}\left(\frac{qT}{n} \right)^{\sigma}d\sigma+\sum_{n\ll \sqrt{qT}}n^{\delta-1}\int_{\frac{1}{2}}^{1}n^{\sigma}d\sigma \\
&\quad +\sqrt{q}\log 2T\int_{\frac{1}{2}}^{1}(qT)^{\frac{\sigma+\delta-1}{2}}d\sigma \ll \sqrt{qT}.
\end{align*}
On the other hand, by (\ref{partialsum}), we get
\begin{align*}
S_2 &\ll (qT)^{-\frac{1}{2}}\sum_{n\leq \frac{qT}{2}}n^{\delta} \int_{1}^{a}\left(\frac{qT}{n} \right)^{\sigma}d\sigma+\sqrt{qT}\int_{1}^{a}\frac{d\sigma}{\sigma} \\
&\ll \sqrt{qT}.
\end{align*}
Hence we complete the proof.
\end{proof}

By (\ref{log-der-bd}) and the above lemma, we get

\begin{align*}
I_3+I_4 &\ll (\log qT)^2 \int_{1-a}^{a}|L'(\sigma+iT,\chi)|d\sigma \\
          &\ll \sqrt{qT}(\log qT)^3.
\end{align*}

Now we consider $I_2$. By the functional equation,
we have

\begin{align*}
&\quad
\frac{L'}{L}(1-a+it,\chi)L'(1-a+it,\chi) \\
&=
\left(\frac{\Delta'}{\Delta}(1-a+it,\chi)-\frac{L'}{L}(a-it,\overline{\chi}) \right) \\
&\quad \times
(\Delta'(1-a+it,\chi)L(a-it,\overline{\chi})-\Delta(1-a+it,\chi)L'(a-it,\overline{\chi})) \\
&=
\frac{\Delta'}{\Delta}(1-a+it,\chi)\Delta'(1-a+it,\chi)L(a-it,\overline{\chi}) \\
&\quad
-2\Delta'(1-a+it,\chi)L'(a-it,\overline{\chi}) \\
&\quad
+\Delta(1-a+it,\chi)\frac{L'}{L}(a-it,\overline{\chi})L'(a-it,\overline{\chi}).
\end{align*}
Thus we can divide $I_2$ into the following three integrals:

\begin{align*}
I_2 &= \frac{1}{2\pi}\int_{T}^{b}\frac{L'}{L}(1-a+it,\chi)L'(1-a+it,\chi)dt \\
     &= \frac{1}{\pi}\int_{b}^{T}\Delta'(1-a+it,\chi)L'(a-it,\overline{\chi})dt \\
     &\quad -\frac{1}{2\pi}\int_{b}^{T}\frac{\Delta'}{\Delta}(1-a+it,\chi)\Delta'(1-a+it,\chi)L(a-it,\overline{\chi})dt \\
     &\quad -\frac{1}{2\pi}\int_{b}^{T}\Delta(1-a+it,\chi)\frac{L'}{L}(a-it,\overline{\chi})L'(a-it,\overline{\chi})dt \\
     &=J_1+J_2+J_3,
\end{align*}
say.
We take complex conjugates of $J_i \ (i=1,2,3)$ to apply Lemma \ref{gonek}.
Then we have

\begin{align*}
\overline{J_1}&=\overline{\frac{1}{\pi}\int_{b}^{T}\Delta'(1-a+it,\chi)L'(a-it,\overline{\chi})dt} \\
&=\frac{1}{\pi}\int_{b}^{T}\Delta'(1-a-it,\overline{\chi})L'(a+it,\chi)dt \\
&=-\frac{1}{\pi}\int_{b}^{T}L'(a+it,\chi)\Delta(1-a-it,\overline{\chi})\log \frac{qt}{2\pi}dt \\
&\quad+O\left(\sum_{n=1}^{\infty}\frac{\log n}{n^a}\int_{b}^{T}\frac{(qt)^{a-\frac{1}{2}}}{t}dt \right) \\
&=\frac{1}{\pi}\int_{b}^{T}\sum_{n=1}^{\infty}\frac{\chi(n)\log n}{n^{a+it}}\Delta(1-a-it,\overline{\chi})\log \frac{qt}{2\pi}dt+O\left((qT)^{a-\frac{1}{2}}(\log qT)^2\right) \\
&=2\frac{\tau(\overline{\chi})}{q}\sum_{1\leq n \leq qT/2\pi} \chi(n)e\left(-\frac{n}{q}\right)(\log n)^2+O\left((qT)^{a-\frac{1}{2}}(\log qT)^3\right).
\end{align*}
On the third equality, we use the approximation (\ref{Stirling-log-der}).
For convenience, we put $x=qT/2\pi$.
By partial summation, the above sum can be calculated as

\begin{align*}
&\quad \sum_{1\leq n \leq x}\chi(n)e\left(-\frac{n}{q}\right)(\log n)^2 \\
&=(\log x)^2 \sum_{m=1}^{q}\chi(m)e\left(-\frac{m}{q}\right)\sum_{\substack{n\leq x \\ n\equiv m \bmod q}}1 \\
&\quad -2\int_{1}^{x}\left(\sum_{m=1}^{q}\chi(m)e\left(-\frac{m}{q}\right)\sum_{\substack{n\leq y \\ n\equiv m \bmod q}}1 \right)\frac{\log y}{y}dy \\
&=\left(\frac{x}{q}\chi(-1)\tau(\chi)+O(\sqrt{q}) \right)(\log x)^2 -2\int_{1}^{x}\left(\frac{y}{q}\chi(-1)\tau(\chi)+O(\sqrt{q}) \right)\frac{\log y}{y}dy \\
&=\frac{\chi(-1)\tau(\chi)}{q}\left(x(\log x)^2-2\int_{1}^{x}\log y dy\right)+O\left(\sqrt{q}(\log x)^2+\sqrt{q}\int_{1}^{x}\frac{\log y}{y}dy \right) \\
&=\frac{\chi(-1)\tau(\chi)}{q}\left(x(\log x)^2-2x\log x+2x\right)+O\left(\sqrt{q}(\log x)^2 \right),
\end{align*}
and we can see that

\begin{equation*}
\frac{\chi(-1)\tau(\chi)\tau(\overline{\chi})}{q^2}=\frac{\overline{\tau(\chi)}\tau(\chi)}{q^2}=\frac{q}{q^2}=\frac{1}{q}.
\end{equation*}
Therefore we obtain

\begin{equation*}
J_1=2\left(\frac{T}{2\pi}\left(\log \frac{qT}{2\pi}\right)^2-\frac{T}{\pi}\log \frac{qT}{2\pi}+\frac{T}{\pi} \right)+O\left((qT)^{a-\frac{1}{2}}(\log qT)^3 \right).
\end{equation*}
Next we consider $J_2$. We have, by (\ref{Stirling-log-der}) again,

\begin{align*}
\overline{J_2}&=\overline{-\frac{1}{2\pi}\int_{b}^{T}\frac{\Delta'}{\Delta}(1-a+it,\chi)\Delta'(1-a+it,\chi)L(a-it,\overline{\chi})dt} \\
&=-\frac{1}{2\pi}\int_{b}^{T}L(a+it,\chi)\frac{\Delta'}{\Delta}(1-a-it,\overline{\chi})\Delta'(1-a-it,\overline{\chi})dt \\
&=-\frac{1}{2\pi}\int_{b}^{T}\sum_{n=1}^{\infty}\frac{\chi(n)}{n^{a+it}}\Delta(1-a-it,\overline{\chi}) \left(\log \frac{qt}{2\pi}\right)^2dt \\
&\quad+O\left(\sum_{n=1}^{\infty}\frac{1}{n^a}\int_{b}^{T}(qt)^{a-\frac{1}{2}}\frac{\log qt}{t}dt \right) \\
&=-\frac{\tau(\overline{\chi})}{q}\sum_{1\leq n \leq qT/2\pi}\chi(n)e\left(-\frac{n}{q} \right)(\log n)^2+O\left((qT)^{a-\frac{1}{2}}(\log qT)^3\right).
\end{align*}
This sum is the same as the previous one. Hence we get
\begin{equation*}
J_2=-\left(\frac{T}{2\pi}\left(\log \frac{qT}{2\pi}\right)^2-\frac{T}{\pi}\log \frac{qT}{2\pi}+\frac{T}{\pi}\right)+O\left((qT)^{a-\frac{1}{2}}(\log qT)^3\right).
\end{equation*}
Finally, we calculate $J_3$. We have

\begin{align*}
\overline{J_3}&=\overline{-\frac{1}{2\pi}\int_{b}^{T}\Delta(1-a+it,\chi)\frac{L'}{L}(a-it,\overline{\chi})L'(a-it,\overline{\chi})dt} \\
&=-\frac{1}{2\pi}\int_{b}^{T}\frac{L'}{L}(a+it,\chi)L'(a+it,\chi)\Delta(1-a-it,\overline{\chi})dt \\
&=-\frac{1}{2\pi}\int_{b}^{T}\left(\sum_{n=1}^{\infty}\frac{\chi(n)\Lambda(n)}{n^{a+it}} \right)\left(\sum_{n=1}^{\infty}\frac{\chi(n)\log n}{n^{a+it}} \right)\Delta(1-a-it,\overline{\chi})dt \\
&=-\frac{\tau(\overline{\chi})}{q}\sum_{1\leq mn \leq qT/2\pi}\chi(mn)e\left(-\frac{mn}{q} \right)\Lambda(m)\log n+O\left((qT)^{a-\frac{1}{2}}(\log qT)^3\right).
\end{align*}
By the orthogonality of Dirichlet characters, we see that

\begin{align*}
&\quad \sum_{mn\leq x}\chi(m)\chi(n)e\left(-\frac{mn}{q} \right)\Lambda(m)\log n \\
&=\sum_{a=1}^{q}\chi(a)\sum_{b=1}^{q}\chi(b)e\left(-\frac{ab}{q} \right)\sum_{\substack{mn \leq x \\ m\equiv a \bmod q \\ n \equiv b \bmod q}} \Lambda(m)\log n \\
&=\frac{1}{\varphi(q)^2} \sum_{\substack{\psi \bmod q \\ \psi' \bmod q}} \sum_{a=1}^{q}\overline{\psi}(a)\chi(a)\sum_{b=1}^{q}\overline{\psi'}(b)\chi(b)e\left(-\frac{ab}{q} \right) \\
&\quad \times \sum_{mn \leq x}\psi(m)\psi'(n)\Lambda(m)\log n.
\end{align*}

We will divide the sum into four parts, according to the following conditions:
\begin{enumerate}
\renewcommand{\labelenumi}{(\roman{enumi})}
    \item $\psi=\psi_0, \ \psi'=\psi'_0$,
    \item $\psi=\psi_0, \ \psi'\neq \psi'_0$,
    \item $\psi \neq \psi_0, \ \psi'=\psi'_0$,
    \item $\psi\neq \psi_0, \ \psi'\neq \psi'_0$,
\end{enumerate}
where $\psi_0=\psi'_0$ is the principal character modulo $q$.
Before discussing further, we will remind some facts on the sum of Dirichlet characters (see \cite[Sec. 8]{Apos}).
We define $G(n,\chi)$ as
\begin{equation*}
G(n,\chi):=\sum_{a=1}^{q}\chi(a)e\left(\frac{an}{q} \right).
\end{equation*}

If a Dirichlet character $\chi \pmod q$ is primitive, then we have
\begin{equation*}
G(a,\chi)=\overline{\chi}(a)\tau(\chi).
\end{equation*}

Now we consider the above four parts.

(i) In this case, we have
\begin{align*}
&\quad \frac{1}{\varphi(q)^2}\sum_{a=1}^{q}\chi(a)\sum_{b=1}^{q}\chi(b)e\left(-\frac{ab}{q} \right)\sum_{mn \leq x}\psi_0(m)\psi_0(n)\Lambda(m)\log n \\
&=\frac{1}{\varphi(q)^2}\sum_{a=1}^{q}\chi(a)G(-a,\chi)\sum_{mn \leq x}\psi_0(m)\psi_0(n)\Lambda(m)\log n \\
&=\frac{\chi(-1)\tau(\chi)}{\varphi(q)}\sum_{mn \leq x}\psi_0(m)\psi_0(n)\Lambda(m)\log n.
\end{align*}

By Perron's formula we get
\begin{align*}
&\quad \sum_{mn \leq x}\psi_0(m)\Lambda(m)\psi_0(n)\log n \\
&=\frac{1}{2\pi i}\int_{a-iU}^{a+iU}\frac{L'}{L}(s,\psi_0)L'(s,\psi_0)\frac{x^s}{s}ds+R,
\end{align*}
where $R$ is the error term appearing in Perron's formula (see \cite[p.140]{Mon&Vau}) and satisfies that
\begin{align*}
R &\ll \sum_{\substack{\frac{x}{2}<mn<2x \\ mn \neq x}} |\Lambda(m)\log n| \min \left(1, \frac{x}{U|x-mn|}\right) \\
   &\quad+\frac{(4x)^a}{U}\sum_{mn=1}^{\infty}\frac{|\Lambda(m)\log n|}{(mn)^a}.
\end{align*}
We will choose an appropriate $U$ later.
The first term of the error term $R$ can be estimated as follows;
\begin{align*}
&\quad \frac{x}{U}\sum_{\frac{x}{2}<mn<x-1}\frac{\Lambda(m) \log n}{x-mn}+\sum_{x-1 \leq mn\leq x+1}\Lambda(m)\log n \\
&\quad +\frac{x}{U}\sum_{x+1<mn<2x}\frac{\Lambda(m)\log n}{mn-x} \\
&\ll \frac{x}{U}\log x\sum_{m<x-1}\frac{\Lambda(m)}{m}\sum_{\frac{x}{2m}<n<\frac{x-1}{m}}\frac{1}{\frac{x}{m}-n} \\
&\quad +(\log x)^2\sum_{x-1\leq l\leq x+1}\sum_{l=mn}1 \\
&\quad +\frac{x}{U}\log x\sum_{m<2x}\frac{\Lambda(m)}{m}\sum_{\frac{x+1}{m}<n<\frac{2x}{m}}\frac{1}{n-\frac{x}{m}} \\
&\ll \frac{x}{U}(\log x)^2\sum_{m<2x}\frac{\Lambda(m)}{m}+(\log x)^2\sum_{x-1\leq l\leq x+1}d(l) \\
&\ll \frac{x}{U}(\log x)^3+x^{\varepsilon},
\end{align*}
where $d(l)$ is the divisor function.
On the last estimates, we use
\begin{equation*}
\sum_{m\leq x}\frac{\Lambda(m)}{m}=\log x+O(1)
\end{equation*}
and
\begin{equation*}
d(x)\ll x^{\varepsilon}.
\end{equation*}
The second is
\begin{equation*}
\ll \frac{(4x)^a}{U}\sum_{mn=1}^{\infty}\frac{\left|\Lambda(m) \log n\right|}{(mn)^a} \ll \frac{x^a}{U}(\log qT)^3.
\end{equation*}

Therefore
\begin{equation*}
R\ll \frac{x}{U}(\log x)^3+x^{\varepsilon}.
\end{equation*}

Since $L(s,\psi_0)=\zeta(s)\prod_{p \mid q}(1-p^{-s})$, there is an absolute constant $C>0$ such that 
\begin{equation*}
L(s,\psi_0)\neq 0 \quad \text{for} \quad \sigma \geq 1-\frac{C}{\log (|t|+2)}
\end{equation*}
(see \cite[p.172]{Mon&Vau}).
With regard to this zero-free region for $L(s,\psi_0)$, let $a'=1-C/\log U$ and $U=\exp \left(4c_1\sqrt{\log qT}\right)$.
By the residue theorem, the integral is

\begin{align*}
&\quad \frac{1}{2\pi i}\int_{a-iU}^{a+iU}\frac{L'}{L}(s,\psi_0)L'(s,\psi_0)\frac{x^s}{s}ds \\
&=\underset{s=1}{\mathrm {Res}} \frac{L'}{L}(s,\psi_0)L'(s,\psi_0)\frac{x^s}{s} \\
&\quad+\frac{1}{2\pi i} \left \{ \int_{a+iU}^{a'+iU}+\int_{a'+iU}^{a'-iU}+\int_{a'-iU}^{a-iU} \right \} \frac{L'}{L}(s,\psi_0)L'(s,\psi_0)\frac{x^s}{s}ds.
\end{align*}
By an argument similar to the proof of Lemma \ref{horizontal}, we can see that the integral on the horizontal line can be estimated as

\begin{align*}
\int_{a\pm iU}^{a'\pm iU}\frac{L'}{L}(s,\psi_0)L'(s,\psi_0)\frac{x^s}{s}ds
&\ll \frac{(\log qU)^3}{U} x^a (qU)^{\frac{3}{16}+\varepsilon}(a-a') \\
&\ll xU^{-\frac{1}{2}}=x\exp \left(-2c_1\sqrt{\log x}\right),
\end{align*}
noting the condition $q\leq \exp \left(c_1\sqrt{\log T}\right) \leq \exp \left(4c_1\sqrt{\log qT}\right)=U$ and (\ref{log-der-bd}).
Since $L'/L(s,\psi_0)\ll |s-1|^{-1}$ and $L'(s,\psi_0)\ll |s-1|^{-2}$ in the neighbourhood around $s=1$,
the integral on the vertical line can be bounded by

\begin{align*}
&\ll x^{a'}(qU)^{\frac{3}{16}+\varepsilon} (\log qU)^3 \int_{-U}^{U}\frac{dt}{1+|t|}+x^{a'}(\log qU)^3\int_{-1}^{1}\frac{dt}{|a'+it|} \\
&\ll x^{a'}(qU)^{\frac{3}{16}+\varepsilon}(\log U)^4 \\
&\ll x^{a'}U^{\frac{1}{2}}=x\exp \left(\left(2c_1-\frac{C}{4c_1} \right)\sqrt{\log x} \right).
\end{align*}
When we put $c_1=\sqrt{C}/4$, we obtain that
\begin{align*}
&\quad \frac{1}{2\pi i}\int_{a-iU}^{a+iU}\frac{L'}{L}(s,\psi_0)L'(s,\psi_0)\frac{x^s}{s}ds \\
&=\underset{s=1}{\mathrm {Res}} \frac{L'}{L}(s,\psi_0)L'(s,\psi_0)\frac{x^s}{s}+O\left(x\exp \left(-\frac{\sqrt{C}}{2}\sqrt{\log x}\right)\right).
\end{align*}

Note that

\begin{align*}
&\quad \underset{s=1}{\mathrm {Res}} \frac{L'}{L}(s,\psi_0)L'(s,\psi_0)\frac{x^s}{s} \\
&=\frac{1}{2!}\lim_{s\rightarrow 1}\frac{d^2}{ds^2}(s-1)^3\frac{L'}{L}(s,\psi_0)L'(s,\psi_0)\frac{x^s}{s}.
\end{align*}
To calculate this residue, we observe that

\begin{align*}
L'(s,\psi_0) &= \zeta'(s)\prod_{p\mid q}(1-p^{-s})+\zeta(s)\left(\prod_{p\mid q}(1-p^{-s}) \right)' \\
                &= \left(-\frac{1}{(s-1)^2}+\sum_{k=1}^{\infty}\gamma_k k(s-1)^{k-1} \right)\prod_{p\mid q}(1-p^{-s}) \\
                &\qquad +\left(\frac{1}{s-1}+\sum_{k=0}^{\infty}\gamma_k (s-1)^k \right)\left(\prod_{p \mid q}(1-p^{-s})\sum_{p\mid q}\frac{\log p}{p^s-1} \right)
\end{align*}
and
\begin{align*}
\frac{L'}{L}(s,\psi_0) &= \frac{\zeta'}{\zeta}(s)+\sum_{p\mid q}\frac{\log p}{p^s-1} \\
                             &= -\frac{1}{s-1}+\sum_{k=0}^{\infty}\eta_k(s-1)^k+\sum_{p\mid q}\frac{\log p}{p^s-1},
\end{align*}
where $\gamma_k$ is the $k$-th Stieltjes constant and can be defined by the limit
\begin{equation*}
\gamma_k=\lim_{n\rightarrow \infty}\left \{ \left(\sum_{m=1}^{n}\frac{(\log m)^k}{m} \right)-\frac{(\log n)^{k+1}}{k+1}\right \},
\end{equation*}
and $\eta_k$ can be represented by the sum
\begin{equation*}
\eta_k=(-1)^k\left \{ \frac{k+1}{k!}\gamma_k+\sum_{n=0}^{k-1}\frac{(-1)^{n-1}}{(k-n-1)!}\eta_n \gamma_{k-n-1} \right \}.
\end{equation*}

Hence we get

\begin{align*}
&\quad \underset{s=1}{\mathrm {Res}} \frac{L'}{L}(s,\psi_0)L'(s,\psi_0)\frac{x^s}{s} \\
&=\frac{1}{2!}\lim_{s\rightarrow 1}\frac{d^2}{ds^2}\prod_{p\mid q}\left(1-\frac{1}{p^s} \right)\frac{x^s}{s} \\
&\quad-\frac{2}{2!}\lim_{s\rightarrow 1}\frac{d}{ds} \prod_{p \mid q}(1-p^{-s})\left( \sum_{p\mid q}\frac{\log p}{p^s-1}+\eta_0+\sum_{p\mid q}\frac{\log p}{p^s-1} \right)\frac{x^s}{s} \\
&\quad -\frac{2}{2!}\lim_{s\rightarrow 1} \prod_{p\mid q}\left(1-\frac{1}{p^s} \right) \left \{ \gamma_1+\gamma_0\sum_{p\mid q}\frac{\log p}{p^s-1}+\eta_1 \right. \\
&\quad \left. -\sum_{p\mid q}\frac{\log p}{p^s-1}\left(\eta_0+\sum_{p\mid q}\frac{\log p}{p^s-1} \right) \right \}\frac{x^s}{s} \\
&=\frac{\varphi(q)}{q}x\left \{ \frac{1}{2}(\log x)^2-\left(\sum_{p\mid q}\frac{\log p}{p-1}+\gamma_0+1 \right)\log x-\frac{1}{2}\left(\sum_{p\mid q}\frac{\log p}{p-1}\right)^2 \right. \\
&\quad \left.+\frac{3}{2}\sum_{p\mid q}p\left(\frac{\log p}{p-1}\right)^2+(1-\gamma_0)\sum_{p\mid q}\frac{\log p}{p-1}+\gamma_0^2+\gamma_0-3\gamma_1+1 \right \}.
\end{align*}
Therefore we can see that
\begin{align*}
&\quad \frac{\chi(-1)\tau(\chi)}{\varphi(q)}\sum_{mn \leq x}\psi_0(m)\psi_0(n)\Lambda(m)\log n \\
&=\frac{\tau(\overline{\chi})}{q}x\left \{ \frac{1}{2}(\log x)^2-\left(\sum_{p\mid q}\frac{\log p}{p-1}+\gamma_0+1 \right)\log x-\frac{1}{2}\left(\sum_{p\mid q}\frac{\log p}{p-1}\right)^2 \right. \\
&\quad \left.+\frac{3}{2}\sum_{p\mid q}p\left(\frac{\log p}{p-1}\right)^2+(1-\gamma_0)\sum_{p\mid q}\frac{\log p}{p-1}+\gamma_0^2+\gamma_0-3\gamma_1+1 \right \} \\
&\quad +O\left(x\exp \left(-\frac{\sqrt{C}}{2}\sqrt{\log x}\right)\right).
\end{align*}
Here we note that $\tau(\chi)/\varphi(q) \ll 1$.

(ii) In the same way, we obtain
\begin{align*}
&\quad \frac{1}{\varphi(q)^2}\sum_{\substack{\psi' \bmod q\\ \psi' \neq \psi_0'}}\sum_{b=1}^{q}\overline{\psi'}(b)\chi(b)\sum_{a=1}^{q}\chi(a)e\left(-\frac{ab}{q} \right) \\
&\quad \times \sum_{mn \leq x}\psi_0(m)\psi'(n)\Lambda(m)\log n \\
&=\frac{\chi(-1)\tau(\chi)}{\varphi(q)^2}\sum_{\psi' \neq \psi_0'}\sum_{b=1}^{q}\overline{\psi'}(b)\sum_{mn \leq x}\psi_0(m)\psi'(n)\Lambda(m)\log n.
\end{align*}
The sum of $\overline{\psi'}$ is $0$. Hence we see that the sum in this case vanishes.

(iii) This case is the same as the case (ii).

(iv)

\begin{align*}
&\quad \frac{1}{\varphi(q)^2}\sum_{\substack{\psi \neq \psi_0 \\ \psi' \neq \psi_0'}}\sum_{a=1}^{q}\overline{\psi}(a)\chi(a)\sum_{b=1}^{q}\overline{\psi'}(b)\chi(b)e\left(-\frac{ab}{q} \right) \\
&\quad \times \sum_{mn \leq x}\psi(m)\psi'(n)\Lambda(m)\log n \\
&=\frac{1}{\varphi(q)^2}\sum_{\substack{\psi \neq \psi_0 \\ \psi' \neq \psi_0'}}\sum_{a=1}^{q}\overline{\psi}(a)\chi(a)\psi'(-a)\overline{\chi}(-a)\tau(\overline{\psi'}\chi) \\
&\quad \times \sum_{mn \leq x}\psi(m)\psi'(n)\Lambda(m)\log n \\
&=\frac{\chi(-1)}{\varphi(q)^2}\sum_{\substack{\psi \neq \psi_0 \\ \psi' \neq \psi_0'}}\sum_{a=1}^{q}\overline{\psi}(a)\psi'(-a)\tau(\overline{\psi'}\chi)\sum_{mn \leq x}\psi(m)\psi'(n)\Lambda(m)\log n \\
&=\frac{\chi(-1)}{\varphi(q)^2}\sum_{\substack{\psi \neq \psi_0 \\ \psi' \neq \psi_0'}}\psi'(-1)\tau(\overline{\psi'}\chi)\sum_{a=1}^{q}\overline{\psi}(a)\psi'(a)\sum_{mn \leq x}\psi(m)\psi'(n)\Lambda(m)\log n \\
&=\frac{\chi(-1)}{\varphi(q)}\sum_{\psi \neq \psi_0}\psi(-1)\tau(\overline{\psi}\chi)\sum_{mn \leq x}\psi(m)\psi(n)\Lambda(m)\log n.
\end{align*}
To show the last equality, we use the fact that the sum over $a$ does not equal to $0$ if and only if $\psi=\psi'$.

In this case, we know the fact that there is an absolute constant $C'>0$ such that
\begin{equation*}
L(s,\chi)\neq 0 \quad \text{for} \quad \sigma >1-\frac{C'}{\log q(|t|+2)}
\end{equation*}
unless $\chi$ is a quadratic character, in which case $L(s,\chi)$ has at most one, necessarily real, zero $\beta <1$ (see \cite[p. 360]{Mon&Vau}).
By the same argument as in the case (i), when we put $c_1=\sqrt{C'}/4$ we have
\begin{align*}
&\quad \sum_{mn \leq x}\psi(m)\Lambda(m)\psi(n)\log n \\
&=-L'(\beta,\psi)\frac{x^{\beta}}{\beta}+O\left(x\exp \left(-\frac{\sqrt{C'}}{2}\sqrt{\log x}\right)\right)
\end{align*}
when $L(s,\psi)$ with a quadratic character $\omega$ has an exceptional zero $\beta$. If there is no exceptional zero, then the first term vanishes.
Hence when $L(s,\omega)$ has an exceptional zero $\beta$ we have
\begin{align*}
&\quad \frac{\chi(-1)}{\varphi(q)}\sum_{\psi \neq \psi_0}\psi(-1)\tau(\overline{\psi}\chi)\sum_{mn \leq x}\psi(m)\psi(n)\Lambda(m)\log n \\
&=-\frac{\chi(-1)}{\varphi(q)}\omega(-1)\tau(\overline{\omega}\chi)L'(\beta,\omega)\frac{x^{\beta}}{\beta}+O\left(\sqrt{q}x\exp \left(-\frac{\sqrt{C'}}{2}\sqrt{\log x}\right)\right),
\end{align*}
otherwise the main term does not appear.

From the above, when we put $c_1=\min \{\sqrt{C}/4,\sqrt{C'}/4 \}$ and $c=c_1/2$, we have
\begin{align*}
J_3&=-\frac{T}{2\pi}\left \{ \frac{1}{2}\left(\log \frac{qT}{2\pi}\right)^2-\left(\sum_{p\mid q}\frac{\log p}{p-1}+\gamma_0+1 \right)\log \frac{qT}{2\pi}-\frac{1}{2}\left(\sum_{p\mid q}\frac{\log p}{p-1}\right)^2 \right. \\
&\quad \left.+\frac{3}{2}\sum_{p\mid q}p\left(\frac{\log p}{p-1}\right)+(1-\gamma_0)\sum_{p\mid q}\frac{\log p}{p-1}+\gamma_0^2+\gamma_0+\gamma_1+1 \right \} \\
&\quad +\frac{\omega \chi(-1)\tau({\overline{\chi})}\tau(\overline{\omega}\chi)}{q\varphi(q)}\frac{L'(\beta,\omega)}{\beta}\left(\frac{qT}{2\pi}\right)^{\beta}+O\left(T\exp \left(-c\sqrt{\log T}\right)\right).
\end{align*}
We note that $\tau(\chi)\sqrt{q}/q\ll 1$.

To complete the proof, we take away the condition on $T$.
When $T$ increases continuously in $|T-\gamma_{\chi}|\ll (\log qT)^{-1}$, the number of relevant $L'(\rho_{\chi},\chi)$ is at most $O(\log qT)$ and the order of each term is $O((qT)^{\frac{3}{16}+\varepsilon})$.
Thus the contribution of these terms is smaller than the error in our main theorem.
Therefore the proof in the unconditional case is completed.

\section{The conditional estimate}
In this section, we assume the GRH.
We choose $a'=1/2+(\log qT)^{-1}$ and $U=qT$.
In the case (i), by Cauchy's theorem,
\begin{align*}
&\quad \frac{1}{2\pi i}\int_{a-iU}^{a+iU}\frac{L'}{L}(s,\psi_0)L'(s,\psi_0)\frac{x^s}{s}ds \\
&=\underset{s=1}{\mathrm {Res}} \frac{L'}{L}(s,\psi_0)L'(s,\psi_0)\frac{x^s}{s} \\
&\quad +\frac{1}{2\pi i} \left \{ \int_{a+iU}^{a'+iU}+\int_{a'+iU}^{a'-iU}+\int_{a'-iU}^{a-iU} \right \} \frac{L'}{L}(s,\psi_0)L'(s,\psi_0)\frac{x^s}{s}ds,
\end{align*}
The integral on the horizontal line is
\begin{align*}
\int_{a'\pm iU}^{a\pm iU}\frac{L'}{L}(s,\psi_0)L'(s,\psi_0)\frac{x^s}{s}ds
\ll
\frac{x^a}{U}(qU)^{\varepsilon}(\log qU)^3
\ll
(qT)^{\varepsilon}.
\end{align*}

As for the vertical line, we note that
\begin{align*}
\frac{L'}{L}(s,\psi_0)
&=
\frac{\zeta'}{\zeta}(s)+\sum_{p\mid q}\frac{\log p}{p^s-1} \ll \log 2q
\end{align*}
for $s=a'+it$ and $0\leq|t|\leq 1$.
Thus
we have
\begin{align*}
&\quad \int_{a'-iU}^{a'+iU}\frac{L'}{L}(s,\psi_0)L'(s,\psi_0)\frac{x^s}{s}ds \\
&=i\int_{-U}^{U} \frac{L'}{L}(a'+it,\psi_0)L'(a'+it,\psi_0)\frac{x^{a'+it}}{a'+it}dt \\
&\ll x^{a'}(\log qU)^3\int_{1}^{U}\frac{(qt)^{\varepsilon}}{t}dt+x^{a'}(\log 2q)^2\int_{-1}^{1}\frac{q^{\varepsilon}}{a'}dt \\
&\ll (qT)^{\frac{1}{2}+\varepsilon}.
\end{align*}

Concerning the case (iv), we can see that
\begin{equation*}
\sum_{nm\leq x}\psi(m)\Lambda(m)\psi(n)\log n\ll (qT)^{\frac{1}{2}+\varepsilon}
\end{equation*}
by the similar argument.
Therefore we can replace the error term in our theorem by $(qT)^{\frac{1}{2}+\varepsilon}$.

\section{The Details of Remark 1}
We consider the case when $q$ is a prime power.
Let $q=p^{\alpha}$, $a'=-(\log qT)^{-1}=1-a$ and $U=qT$.
In the case (i), by the residue theorem
\begin{align*}
&\quad \frac{1}{2\pi i}\int_{a-iU}^{a+iU}\frac{L'}{L}(s,\psi_0)L'(s,\psi_0)\frac{x^s}{s}ds \\
&=\underset{s=1}{\mathrm {Res}} \frac{L'}{L}(s,\psi_0)L'(s,\psi_0)\frac{x^s}{s}+\underset{s=0}{\mathrm {Res}} \frac{L'}{L}(s,\psi_0)L'(s,\psi_0)\frac{x^s}{s} \\
&\quad+\sum_{\substack{\rho \neq 0 \\ |\Im \rho|\leq U}} L'(\rho,\psi_0)\frac{x^{\rho}}{\rho} \\
&\quad+\frac{1}{2\pi i} \left \{ \int_{a+iU}^{a'+iU}+\int_{a'+iU}^{a'-iU}+\int_{a'-iU}^{a-iU} \right \} \frac{L'}{L}(s,\psi_0)L'(s,\psi_0)\frac{x^s}{s}ds,
\end{align*}
where $\rho$ runs over the zeros of $L(s,\psi_0)$.
With regard to the residue at $s=0$, we can see that
\begin{equation*}
\frac{L'}{L}(s,\psi_0)=\frac{\zeta'}{\zeta}(s)+\frac{\log p}{p^s-1} \quad (q=p^{\alpha})
\end{equation*}
and
\begin{equation*}
\frac{\log p}{p^s-1}=\frac{1}{s}\cdot \frac{s\log p}{e^{s\log p}-1}=\frac{1}{s}\sum_{n=0}^{\infty}\frac{B_n}{n!}(s\log p)^n,
\end{equation*}
where $B_n$ is the $n$-th Bernoulli number,
and hence we have
\begin{align*}
&\quad \underset{s=0}{\mathrm {Res}} \frac{L'}{L}(s,\psi_0)L'(s,\psi_0)\frac{x^s}{s} \\
&=\lim_{s\rightarrow 0}\frac{d}{ds}s\frac{L'}{L}(s,\psi_0)L'(s,\psi_0)x^s \\
&=\lim_{s\rightarrow 0}\frac{d}{ds}s\left(\frac{\zeta'}{\zeta}(s)+\frac{1}{s}\sum_{n=0}^{\infty}\frac{B_n}{n!}(s\log p)^n \right)L'(s,\psi_0)x^s \\
&=L''(0,\psi_0)+\left(\frac{\zeta'}{\zeta}(0)+B_1\log p +\log x \right)L'(0,\psi_0) \\
&=3\zeta'(0)\log p-\frac{3}{2}\zeta(0)(\log p)^2+\zeta(0)\log x \ll (\log qT)^2.
\end{align*}
The integral on the horizontal line is
\begin{align*}
&\quad \int_{1-a\pm iU}^{a\pm iU}\frac{L'}{L}(s,\psi_0)L'(s,\psi_0)\frac{x^s}{s}ds \\
&\ll \left \{ \int_{\frac{1}{2}\pm iU}^{a\pm iU}+\int_{1-a\pm iU}^{\frac{1}{2}\pm iU} \right \} \frac{L'}{L}(s,\psi_0)L'(s,\psi_0)\frac{x^s}{s}ds \\
&\ll \frac{x^a}{U}(qU)^{\varepsilon}(\log qU)^3+\frac{x^{\frac{1}{2}}}{U}(qU)^{\frac{1}{2}}(\log qU)^4 \\
&\ll (qU)^{\varepsilon}(\log qU)^3+\sqrt{q}(\log qU)^4.
\end{align*}

On the integral along the vertical line, since $|s-\rho_{\psi_0}|\gg 1$, by (\ref{log-der-sum}), we can see that
\begin{equation*}
\frac{L'}{L}(s,\psi_0)\ll \log q(|t|+2).
\end{equation*}
Therefore we have
\begin{align*}
&\quad \int_{1-a-iU}^{1-a+iU}\frac{L'}{L}(s,\psi_0)L'(s,\psi_0)\frac{x^s}{s}ds \\
&=i\int_{-U}^{U}\frac{L'}{L}(1-a+it,\psi_0)L'(1-a+it,\psi_0)\frac{x^{1-a+it}}{1-a+it}dt \\
&\ll (\log qU)^2\left|\int_{-U}^{U}\zeta(1-a+it)\frac{dt}{1-a+it} \right| \\
&\ll (\log qU)^2\left(\log U\int_{1}^{U} t^{-\frac{1}{2}}dt+\int_{-1}^{1}\frac{dt}{|1-a+it|} \right) \\
&\ll \sqrt{U}(\log qU)^3.
\end{align*}
Here we use the well-known estimate
\begin{equation*}
\zeta(s) \ll (|t|+2)^{\frac{1}{2}}\log (|t|+2) \quad \text{for} \quad -\frac{1}{\log T}\leq \sigma < \frac{1}{2}.
\end{equation*}

The sum over $\rho$ consists of two sums as
\begin{align*}
\sum_{\substack{\rho \neq 0 \\ |\Im \rho|\leq U}} L'(\rho,\psi_0)\frac{x^{\rho}}{\rho}
&=\sum_{|\gamma|\leq U}L'\left(\frac{1}{2}+i\gamma,\psi_0 \right)\frac{x^{\frac{1}{2}+i\gamma}}{\frac{1}{2}+i\gamma} \\
&\quad +\sum_{\substack{\left|\frac{2\pi k}{\log p}\right|\leq U \\ k\neq 0}}L'\left(\frac{2\pi ik}{\log p},\psi_0 \right)\frac{x^{\frac{2\pi k}{\log p}i}\log p}{2\pi ik} \\
&=S_1+S_2,
\end{align*}
say.
Since
\begin{equation*}
L'\left(\frac{1}{2}+i\gamma,\psi_0\right)=\zeta'\left(\frac{1}{2}+i\gamma \right)(1-p^{-\frac{1}{2}-i\gamma}),
\end{equation*}
we have
\begin{align*}
S_1 &\ll x^{\frac{1}{2}}\sum_{\gamma \leq U}\frac{\left|L'\left(\frac{1}{2}+i\gamma,\psi_0 \right) \right|}{\gamma} \\
&\ll x^{\frac{1}{2}}\left(\sum_{\gamma \leq U}\frac{\left|\zeta'\left(\frac{1}{2}+i\gamma \right) \right|^2}{\gamma}\right)^{\frac{1}{2}}\left(\sum_{\gamma \leq U}\frac{1}{\gamma} \right)^{\frac{1}{2}} \\
&\ll x^{\frac{1}{2}}(\log U)^{\frac{7}{2}}
\end{align*}
by partial summation and the fact that
\begin{equation}\label{zeta'-upper}
\sum_{0<\gamma \leq T}\left|\zeta' \left(\frac{1}{2}+i\gamma \right) \right|^2 \asymp T(\log T)^4
\end{equation}
proved by Gonek \cite{Go}.

On the other hand, since
\begin{align*}
L'\left(\frac{2\pi ik}{\log p},\psi_0 \right)
=
\zeta \left(\frac{2\pi ik}{\log p} \right)\log p,
\end{align*}
we see that
\begin{align*}
S_2 \ll (\log p)^2 \sum_{\frac{2\pi k}{\log p}\leq U}\frac{\left|\zeta \left(\frac{2\pi ik}{\log p} \right)\right|}{2\pi k} \ll \sqrt{U}\log U (\log q)^2
\end{align*}
by the estimate
\begin{equation*}
\zeta(s) \ll (|t|+2)^{\frac{1}{2}}\log (|t|+2) \quad \text{for} \quad -\frac{1}{\log T}\leq \sigma < \frac{1}{2}
\end{equation*}
again.
Therefore we can see that
\begin{align*}
&\quad \frac{\chi(-1)\tau(\chi)}{\varphi(q)}\sum_{mn \leq x}\psi_0(m)\psi_0(n)\Lambda(m)\log n \\
&=\frac{\tau(\overline{\chi})}{q}x\left \{ \frac{1}{2}(\log x)^2-\left(\sum_{p\mid q}\frac{\log p}{p-1}+\gamma_0+1 \right)\log x-\frac{1}{2}\left(\sum_{p\mid q}\frac{\log p}{p-1}\right)^2 \right. \\
&\quad \left.+\frac{3}{2}\sum_{p\mid q}p\left(\frac{\log p}{p-1}\right)^2+(1-\gamma_0)\sum_{p\mid q}\frac{\log p}{p-1}+\gamma_0^2+\gamma_0-3\gamma_1+1 \right \} \\
&\quad +O\left(x^{\frac{1}{2}}(\log U)^{\frac{7}{2}} \right).
\end{align*}

As for the case (iv), we need to deal with the Dirichlet $L$-functions with primitive and also imprimitive characters. However, it is sufficient to consider these with only primitive characters, for we put $q=p^{\alpha}$. For primitive characters, the integral on the vertical line can be estimated as
\begin{align*}
&\quad \int_{1-a-iU}^{1-a+iU}\frac{L'}{L}(s,\psi)L'(s,\psi)\frac{x^s}{s}ds \\
&=\int_{1-a-iU}^{1-a+iU} \Delta(s,\psi)\left \{ \left(\frac{\Delta'}{\Delta}(s,\psi) \right)^2L(1-s,\overline{\psi}) \right. \\
&\left. \quad -2\frac{\Delta'}{\Delta}(s,\psi)L'(1-s,\overline{\psi})+\frac{L'}{L}(1-s,\overline{\psi})L'(1-s,\overline{\psi})\right \} \frac{x^s}{s}ds \\
&\ll q^{a-\frac{1}{2}}\left|\int_{-U}^{U}\left(t^{a-\frac{1}{2}}\exp \left(it\log \frac{2\pi e}{qt}\right)+O(t^{a-\frac{3}{2}})\right)\right. \\
&\quad \left. \times \left((\log qU)^2 L(a-it,\psi)+\frac{L'}{L}(a-it,\psi)L'(a-it,\psi)\right)\frac{x^{1-a+it}}{1-a+it}dt\right| \\
&\ll x^{1-a}q^{a-\frac{1}{2}}\left((\log U)^2 \sum_{n=1}^{\infty}\frac{1}{n^a}\left|\int_{1}^{U}\left(t^{a-\frac{3}{2}}\exp \left(it\log \frac{2\pi exn}{qt}\right)+O(t^{a-\frac{5}{2}})\right)dt\right| \right. \\
&\quad \left. +\sum_{m=2}^{\infty}\frac{\Lambda(m)}{m^a}\sum_{n=1}^{\infty}\frac{\log n}{n^a}\left|\int_{1}^{U}\left(t^{a-\frac{3}{2}}\exp \left(it\log \frac{2\pi exmn}{qt}\right)+O(t^{a-\frac{5}{2}})\right)dt\right| \right) \\
&\quad +O(q^{a-\frac{1}{2}}(\log U)^3).
\end{align*}
Since
\begin{equation*}
\frac{d^2}{dt^2}\left(t\log \frac{2\pi exn}{qt}\right)=-t^{-1},
\end{equation*}
by the second derivative test,
\begin{align*}
&\quad \int_{1}^{U}t^{a-\frac{3}{2}}\exp \left(it\log \frac{2\pi exn}{qt}\right)dt \\
&\ll \sum_{l\leq [\log U]+1} \int_{\frac{U}{2^l}}^{\frac{U}{2^{l-1}}}t^{a-\frac{3}{2}}\exp \left(it\log \frac{2\pi exn}{qt}\right)dt \\
&\ll \sum_{l\leq [\log U]+1} 1 \ll \log U.
\end{align*}
Therefore we obtain
\begin{equation*}
\int_{1-a-iU}^{1-a+iU}\frac{L'}{L}(s,\psi)L'(s,\psi)\frac{x^s}{s}ds \ll q^{a-\frac{1}{2}}(\log U)^4.
\end{equation*}

On the sum $S_1$, we assume the estimate (\ref{L'-upper}).
By partial summation and this assumption, we have
\begin{align*}
S_1 &\ll x^{\frac{1}{2}}\sum_{0<\gamma_{\psi} \leq U}\frac{\left|L'\left(\frac{1}{2}+i\gamma_{\psi},\psi \right) \right|}{\gamma_{\psi}} \\
&\ll x^{\frac{1}{2}}\left(\sum_{0<\gamma_{\psi} \leq U}\frac{\left|L'\left(\frac{1}{2}+i\gamma_{\psi},\psi \right) \right|^2}{\gamma_{\psi}}\right)^{\frac{1}{2}}\left(\sum_{0<\gamma_{\psi} \leq U}\frac{1}{\gamma_{\psi}} \right)^{\frac{1}{2}} \\
&\ll x^{\frac{1}{2}}(\log U)^{\frac{7}{2}}.
\end{align*}
On the other hand, the counterpart of the sum $S_2$ does not appear. When $\psi \pmod q$ is induced by $\psi^{\star} \pmod d$ with $d\mid q$, we see that
\begin{equation*}
L(s,\psi)=L(s,\psi^{\star})\prod_{\substack{p\mid q\\ p\nmid d}}\left(1-\frac{\psi^{\star}(p)}{p^s} \right).
\end{equation*}
However we assume that $q=p^{\alpha}$. Thus the products on the right-hand side is $1$.
Hence there is no zeros on the imaginary axis.

Therefore we can replace the estimate of the error term by
\begin{equation*}
\sqrt{qT}(\log qT)^{\frac{7}{2}}.
\end{equation*}

\section*{Acknowledgement}
I would like to thank my supervisor Professor Kohji Matsumoto for useful advice.
I am grateful to the seminar members for some helpful remarks and discussions.

\end{document}